\documentclass[reqno]{amsart}
\usepackage{pictexwd,dcpic}
\usepackage{slashbox}

\newtheorem{theorem}{Theorem}[section]
\newtheorem{proposition}[theorem]{Proposition}

\newtheorem{corollary}[theorem]{Corollary}

\theoremstyle{definition}

\newtheorem{example}[theorem]{Example}

\theoremstyle{remark}
\newtheorem{remark}[theorem]{Remark}

\numberwithin{equation}{section}

\begin{document}

\title{Algebraic integers as special values of modular units}

\author{Ja Kyung Koo}
\address{Department of Mathematical Sciences, KAIST}
\curraddr{Daejeon 373-1, Korea} \email{jkkoo@math.kaist.ac.kr}
\thanks{}

\author{Dong Hwa Shin}
\address{Department of Mathematical Sciences, KAIST}
\curraddr{Daejeon 373-1, Korea} \email{shakur01@kaist.ac.kr}
\thanks{}

\author{Dong Sung Yoon}
\address{Department of Mathematical Sciences, KAIST}
\curraddr{Daejeon 373-1, Korea} \email{yds1850@kaist.ac.kr}
\thanks{}


\keywords{Dedekind eta-function, Siegel functions
\newline The research was partially supported by Basic Science Research Program through the NRF of Korea
funded by MEST (2010-0001654).}

\begin{abstract}
Let $\varphi(\tau)=\eta((\tau+1)/2)^2/\sqrt{2\pi}e^\frac{\pi
i}{4}\eta(\tau+1)$ where $\eta(\tau)$ is the Dedekind eta-function.
We show that if $\tau_0$ is an imaginary quadratic number with
$\mathrm{Im}(\tau_0)>0$ and $m$ is an odd integer, then
$\sqrt{m}\varphi(m\tau_0)/\varphi(\tau_0)$ is an algebraic integer
dividing $\sqrt{m}$. This is a generalization of Theorem 4.4 given
in \cite{B-C-Z}. On the other hand, let $K$ be an imaginary
quadratic field and $\theta_K$ be an element of $K$ with
$\mathrm{Im}(\theta_K)>0$ which generators the ring of integers of
$K$ over $\mathbb{Z}$. We develop a sufficient condition of $m$ for
$\sqrt{m}\varphi(m\theta_K)/\varphi(\theta_K)$ to become a unit.
\end{abstract}

\maketitle

\section{Introduction}

The \textit{Dedekind eta-function} $\eta(\tau)$ is defined to be the
following infinite product expansion
\begin{equation}\label{eta}
\eta(\tau)=\sqrt{2\pi}e^\frac{\pi
i}{4}q^\frac{1}{24}\prod_{n=1}^\infty(1-q^n)\quad(\tau\in\mathfrak{H})
\end{equation}
where $q=e^{2\pi i\tau}$ and
$\mathfrak{H}=\{\tau\in\mathbb{C}:\mathrm{Im}(\tau)>0\}$. Define a
function
\begin{equation}\label{phi}
\varphi(\tau)=\frac{1}{\sqrt{2\pi}e^\frac{\pi
i}{4}}\frac{\eta((\tau+1)/2)^2}{\eta(\tau+1)}= \prod_{n=1}^\infty
(1+q^{n-\frac{1}{2}})^2(1-q^n) \quad(\tau\in\mathfrak{H}).
\end{equation}
Motivated by Ramanujan's evaluation of $\varphi(mi)/\varphi(i)$ for
some positive integers $m$ (\cite{Ramanujan}), which are algebraic
numbers, Berndt-Chan-Zhang proved the following theorem.

\begin{theorem}[\cite{B-C-Z} Theorem 4.4]\label{BCZ}
Let $m$ and $n$ be positive integers. If $m$ is odd, then
$\sqrt{2m}\varphi(mni)/\varphi(ni)$ is an algebraic integer dividing
$2\sqrt{m}$, while if $m$ is even, then
$2\sqrt{m}\varphi(mni)/\varphi(ni)$ is an algebraic integer dividing
$4\sqrt{m}$.
\end{theorem}

In this paper we shall revisit the above theorem and improve the
statement when $m$ is odd, as follows:

\begin{theorem}\label{main}
Let $m$ be a positive integer and $\tau_0\in\mathfrak{H}$ be
imaginary quadratic. Then
$2\sqrt{m}\varphi(m\tau_0)/\varphi(\tau_0)$ is an algebraic integer
dividing $4\sqrt{m}$. In particular, if $m$ is odd, then
$\sqrt{m}\varphi(m\tau_0)/\varphi(\tau_0)$ is an algebraic integer
dividing $\sqrt{m}$.
\end{theorem}

For $(r_1,r_2)\in\mathbb{Q}^2-\mathbb{Z}^2$, the \textit{Siegel
function} $g_{(r_1,r_2)}(\tau)$ is defined by
\begin{equation}\label{Siegel}
g_{(r_1,r_2)}(\tau)=-q^{\frac{1}{2}\mathbf{B}_2(r_1)} e^{\pi
ir_2(r_1-1)}(1-q_z)\prod_{n=1}^\infty(1-q^nq_z)(1-q^nq_z^{-1})\qquad(\tau\in\mathfrak{H})
\end{equation}
where $\mathbf{B}_2(x)=x^2-x+1/6$ is the second Bernoulli polynomial
and $q_z=e^{2\pi iz}$ with $z=r_1\tau+r_2$. We shall first express
the function $\varphi(m\tau)/\varphi(\tau)$ as a product of certain
eta-quotient and Siegel functions (Proposition
\ref{phiexpression}(i)). Then, we shall prove Theorem \ref{main} in
$\S$\ref{proof1} by using integrality of Siegel functions over
$\mathbb{Z}[j(\tau)]$ (Proposition \ref{integrality}) where
\begin{equation*}
j(\tau)= \bigg(\frac{\eta(\tau)^{24}+2^8\eta(2\tau)^{24}}
{\eta(\tau)^{16}\eta(2\tau)^8}\bigg)^3=q^{-1}+744+196884q
+21493760q^2+\cdots
\end{equation*}
is the well-known \textit{modular $j$-function} (\cite{Cox} Theorem
12.17).
\par
On the other hand, let $K$ be an imaginary quadratic field with
discriminant $d_K$, and define
\begin{eqnarray}\label{theta}
\theta_K=\left\{\begin{array}{ll}\frac{\sqrt{d_K}}{2}&\textrm{for}~d_K\equiv0\pmod{4}\\
\frac{-1+\sqrt{d_K}}{2}&\textrm{for}~
d_K\equiv1\pmod{4},\end{array}\right.
\end{eqnarray}
which generates the ring of integers of $K$ over $\mathbb{Z}$.
Ramachandra showed that if $N$ ($\geq2$) is an integer with more
than one prime ideal factor $K$, then
$g_{(0,\frac{1}{N})}(\theta_K)^{12N}$ is a unit (Proposition
\ref{Rama}). This fact, together with the Shimura's reciprocity law
(Proposition \ref{Stev}), will enable us to prove the following
theorem in $\S$\ref{proof2}.

\begin{theorem}\label{unitcriterion}
If $m$ ($\geq3$) is an odd integer whose prime factors split in $K$,
then $\sqrt{m}\varphi(m\theta_K)/\varphi(\theta_K)$ is a unit.
\end{theorem}

\section{Arithmetic properties of Siegel functions}

In this section we shall examine some arithmetic properties of
Siegel functions. For the classical theory of modular functions, one
can refer to \cite{Lang} or \cite{Shimura}.
\par
For each positive integer $N$, let $\zeta_N=e^\frac{2\pi i}{N}$ and
$\mathcal{F}_N$ be the field of meromorphic modular functions of
level $N$ whose Fourier coefficients belong to the $N^\textrm{th}$
cyclotomic field $\mathbb{Q}(\zeta_N)$.

\begin{proposition}
For each positive integer $N$, $\mathcal{F}_N$ is a Galois extension
of $\mathcal{F}_1=\mathbb{Q}(j(\tau))$ whose Galois group is
isomorphic to
\begin{equation*}
\mathrm{GL}_2(\mathbb{Z}/N\mathbb{Z})/\{\pm1_2\}=G_N\cdot
\mathrm{SL}_2(\mathbb{Z}/N\mathbb{Z})/\{\pm1_2\}
\end{equation*}
where
\begin{equation*}
G_N=\bigg\{\begin{pmatrix}1&0\\0&d\end{pmatrix}
~:~d\in(\mathbb{Z}/N\mathbb{Z})^*\bigg\}.
\end{equation*}
Here, the matrix $\begin{pmatrix}1&0\\0&d\end{pmatrix}\in G_N$ acts
on $\sum_{n=-\infty}^\infty c_n q^\frac{n}{N}\in\mathcal{F}_N$ by
\begin{equation*}
\sum_{n=-\infty}^\infty c_nq^\frac{n}{N}\mapsto
\sum_{n=-\infty}^\infty c_n^{\sigma_d}q^\frac{n}{N}
\end{equation*}
where $\sigma_d$ is the automorphism of $\mathbb{Q}(\zeta_N)$
induced by $\zeta_N\mapsto\zeta_N^d$. And, for an element
$\gamma\in\mathrm{SL}_2(\mathbb{Z}/N\mathbb{Z})/\{\pm1_2\}$ let
$\gamma'\in\mathrm{SL}_2(\mathbb{Z})$ be a preimage of $\gamma$ via
the natural surjection
$\mathrm{SL}_2(\mathbb{Z})\rightarrow\mathrm{SL}_2(\mathbb{Z}/N\mathbb{Z})/\{\pm1_2\}$.
Then $\gamma$ acts on $h\in\mathcal{F}_N$ by composition
\begin{equation*}
h\mapsto h\circ\gamma'
\end{equation*}
as linear fractional transformation.
\end{proposition}
\begin{proof}
See \cite{Lang} Chapter 6 Theorem 3.
\end{proof}
\begin{proposition}\label{integrality}
Let $(r_1,r_2)\in\frac{1}{N}\mathbb{Z}^2-\mathbb{Z}^2$ for some
integer $N\geq2$.
\begin{itemize}
\item[(i)] $g_{(r_1,r_2)}(\tau)$ is integral over
$\mathbb{Z}[j(\tau)]$. Namely, $g_{(r_1,r_2)}(\tau)$ is a zero of a
monic polynomial whose coefficients are in $\mathbb{Z}[j(\tau)]$.
\item[(ii)] Suppose that $(r_1,r_2)$ has the primitive
denominator $N$ (that is, $N$ is the smallest positive integer such
that $(Nr_1,Nr_2)\in\mathbb{Z}^2$). If $N$ is composite (that is,
$N$ has at least two prime factors), then $g_{(r_1,r_2)}(\tau)^{-1}$
is also integral over $\mathbb{Z}[j(\tau)]$.
\item[(iii)] $g_{(r_1,r_2)}(\tau)$ is holomorphic and has no
zeros and poles on $\mathfrak{H}$. Furthermore,
$g_{(r_1,r_2)}(\tau)$ (respectively,
$g_{(r_1,r_2)}(\tau)^{12N/\gcd(6,N)})$ belongs to
$\mathcal{F}_{12N^2}$ (respectively, $\mathcal{F}_N$).
\end{itemize}
\end{proposition}
\begin{proof}
See \cite{K-S} $\S$3, \cite{K-L} Chapter 2 Theorems 2.2, 1.2 and
Chapter 3 Theorem 5.2.
\end{proof}

\begin{remark}
Let $g(\tau)$ be an element of $\mathcal{F}_N$ for some positive
integer $N$. If both $g(\tau)$ and $g(\tau)^{-1}$ are integral over
$\mathbb{Q}[j(\tau)]$ (respectively, $\mathbb{Z}[j(\tau)]$), then
$g(\tau)$ is called a \textit{modular unit} (respectively,
\textit{modular unit over $\mathbb{Z}$}) of level $N$. As is
well-known, $g(\tau)$ is a modular unit if and only if it has no
zeros and poles on $\mathfrak{H}$ (\cite{K-L} Chapter 2 $\S$2 or
\cite{K-S} $\S$2). Hence $g_{(r_1,r_2)}(\tau)$ is a modular unit for
any $(r_1,r_2)\in\mathbb{Q}^2-\mathbb{Z}^2$ by (iii). Moreover, if
$(r_1,r_2)$ has the composite primitive denominator, then
$g_{(r_1,r_2)}(\tau)$ is a modular unit over $\mathbb{Z}$ by (ii).
\end{remark}

We recall some basic transformation formulas of Siegel functions.

\begin{proposition}\label{transform}
Let $r=(r_1,r_2)\in\mathbb{Q}^2-\mathbb{Z}^2$.
\begin{itemize}
\item[(i)] We have
\begin{equation*}
g_{-r}(\tau)=g_{(-r_1,-r_2)}(\tau)=-g_r(\tau).
\end{equation*}
\item[(ii)] For
$S=\begin{pmatrix}0&-1\\1&0\end{pmatrix}$ and
$T=\begin{pmatrix}1&1\\0&1\end{pmatrix}$ we get
\begin{eqnarray*}
\begin{array}{ccccc}
g_r(\tau)\circ S&=&\zeta_{12}^9 g_{rS}(\tau)&=&\zeta_{12}^9 g_{(r_2,-r_1)}(\tau)\\
g_r(\tau)\circ
T&=&\zeta_{12}g_{rT}(\tau)&=&\zeta_{12}g_{(r_1,r_1+r_2)}(\tau).
\end{array}
\end{eqnarray*}
Hence we obtain that for any $\gamma\in\mathrm{SL}_2(\mathbb{Z})$,
\begin{equation*}
g_r(\tau)\circ\gamma =\varepsilon g_{r\gamma}(\tau)
\end{equation*}
with $\varepsilon$ a $12^\textrm{th}$ root of unity (depending on
$\gamma$).
\item[(iii)] For $s=(s_1,s_2)\in\mathbb{Z}^2$ we have
\begin{equation*}
g_{r+s}(\tau) =g_{(r_1+s_1,r_2+s_2)}(\tau)
=(-1)^{s_1s_2+s_1+s_2}e^{-\pi i(s_1r_2-s_2r_1)}g_r(\tau).
\end{equation*}
\item[(iv)] Let $r\in\frac{1}{N}\mathbb{Z}^2-\mathbb{Z}^2$ for some integer $N\geq2$.
Each element $\alpha=\begin{pmatrix}a&b\\c&d\end{pmatrix}$ in
$\mathrm{GL}_2(\mathbb{Z}/N\mathbb{Z})/\{\pm1_2\}\simeq\mathrm{Gal}(\mathcal{F}_N/
\mathcal{F}_1)$ acts on $g_r(\tau)^{12N/\gcd(6,N)}$ by
\begin{equation*}
\bigg(g_r(\tau)^\frac{12N}{\gcd(6,N)}\bigg)^\alpha=
g_{r\alpha}(\tau)^\frac{12N}{\gcd(6,N)}
=g_{(r_1a+r_2c,r_1b+r_2d)}(\tau)^\frac{12N}{\gcd(6,N)}.
\end{equation*}
\item[(v)] We have the order formula
\begin{equation*}
\mathrm{ord}_q g_r(\tau)=\frac{1}{2}\mathbf{B}_2 (\langle
r_1\rangle)
\end{equation*}
where $\langle r_1\rangle$ is the fractional part of $r_1$ in the
interval $[0,1)$.
\end{itemize}
\end{proposition}
\begin{proof} See \cite{K-S} Propositions 2.4, 2.5
and \cite{K-L} p. 31.
\end{proof}
\begin{remark}
The expression $r\alpha$ in (iv) is well-defined by (i) and (iii).
\end{remark}

\begin{proposition}\label{phiexpression}
\begin{itemize}
\item[(i)] We can express $\varphi(\tau)$ as
\begin{equation*}
\varphi(\tau)=-\frac{1}{\sqrt{2\pi}}\eta(\tau)g_{(\frac{1}{2},\frac{1}{2})}(\tau).
\end{equation*}
\item[(ii)] We dervie
\begin{equation*}
g_{(0,\frac{1}{2})}(\tau) g_{(\frac{1}{2},0)}(\tau)
g_{(\frac{1}{2},\frac{1}{2})}(\tau)=2e^\frac{\pi i}{4}.
\end{equation*}
\item[(iii)] If $m$ ($\geq3$) is an odd integer, then we have the relation
\begin{equation*}
\frac{g_{(\frac{1}{2},\frac{1}{2})}(m\tau)}{g_{(\frac{1}{2},\frac{1}{2})}(\tau)}
=(-1)^\frac{m-1}{2}\prod_{k=1}^{m-1}g_{(\frac{1}{2},\frac{1}{2}+\frac{k}{m})}(\tau).
\end{equation*}
\end{itemize}
\end{proposition}
\begin{proof}
(i) By the definition (\ref{Siegel}) we have
\begin{equation*}
g_{(\frac{1}{2},\frac{1}{2})}(\tau)=-q^{\frac{1}{2}\mathbf{B}_2(\frac{1}{2})}e^{-\frac{\pi
i}{4}}(1+q^\frac{1}{2})
\prod_{n=1}^\infty(1+q^{n+\frac{1}{2}})(1+q^{n-\frac{1}{2}})=-e^{-\frac{\pi
i}{4}}q^{-\frac{1}{24}} \prod_{n=1}^\infty(1+q^{n-\frac{1}{2}})^2.
\end{equation*}
One can readily obtain the assertion by the definition (\ref{eta})
of $\eta(\tau)$ and the infinite product expansion (\ref{phi}) of
$\varphi(\tau)$.\\
(ii) Put $g(\tau)=g_{(0,\frac{1}{2})}(\tau)
g_{(\frac{1}{2},0)}(\tau) g_{(\frac{1}{2},\frac{1}{2})}(\tau)$,
which is an element of $\mathcal{F}_{48}$ by Proposition
\ref{integrality}(iii). For any
$\alpha=\begin{pmatrix}a&b\\c&d\end{pmatrix}\in\mathrm{SL}_2(\mathbb{Z})$
we derive that
\begin{eqnarray*}
\mathrm{ord}_q (g(\tau)\circ\alpha)&=& \mathrm{ord}_q\bigg(
g_{(\frac{c}{2},\frac{d}{2})}(\tau)
g_{(\frac{a}{2},\frac{b}{2})}(\tau)
g_{(\frac{a+c}{2},\frac{b+d}{2})}(\tau)\bigg)~\textrm{by Proposition \ref{transform}(ii)}\\
&=&\frac{1}{2}\mathbf{B}_2\bigg(\bigg\langle\frac{c}{2}\bigg\rangle\bigg)
+\frac{1}{2}\mathbf{B}_2\bigg(\bigg\langle\frac{a}{2}\bigg\rangle\bigg)
+\frac{1}{2}\mathbf{B}_2\bigg(\bigg\langle\frac{a+c}{2}\bigg\rangle\bigg)
~\textrm{by Proposition \ref{transform}(v)}\\
&=&\frac{1}{2}\mathbf{B}_2(0)+2\cdot\frac{1}{2}\mathbf{B}_2\bigg(\frac{1}{2}\bigg)~
\textrm{because both $a$ and $c$ cannot be even}\\
&=&0.
\end{eqnarray*}
This observation implies that $g(\tau)$ is holomorphic at every
cusp. Thus $g(\tau)$ is a holomorphic function on the modular curve
of level $48$ (which is a compact Riemann surface, or an algebraic
curve); and hence it must be a constant. It follows that
\begin{eqnarray*}
g(\tau)&=&-2e^{-\frac{3\pi i}{4}}\prod_{n=1}^\infty
(1+q^n)^2(1-q^{n-\frac{1}{2}})^2(1+q^{n-\frac{1}{2}})^2~\textrm{by
the definition (\ref{Siegel})}\\
&=&-2e^{-\frac{3\pi i}{4}}\lim_{q\rightarrow0}\prod_{n=1}^\infty
(1+q^n)^2(1-q^{n-\frac{1}{2}})^2(1+q^{n-\frac{1}{2}})^2\\
&=&2e^\frac{\pi i}{4}.
\end{eqnarray*}
(iii) By the definition (\ref{Siegel}) we have
\begin{eqnarray*}
\frac{g_{(\frac{1}{2},\frac{1}{2})}(m\tau)}{g_{(\frac{1}{2},\frac{1}{2})}(\tau)}
&=&\frac{-q^{\frac{m}{2}\mathbf{B}_2(\frac{1}{2})}e^{-\frac{\pi
i}{4}}(1+q^\frac{m}{2})\prod_{n=1}^\infty(1+q^{mn+\frac{m}{2}})(1+q^{mn-\frac{m}{2}})}
{-q^{\frac{1}{2}\mathbf{B}_2(\frac{1}{2})}e^{-\frac{\pi
i}{4}}(1+q^\frac{1}{2})\prod_{n=1}^\infty(1+q^{n+\frac{1}{2}})(1+q^{n-\frac{1}{2}})}\\
&=&q^{\frac{1-m}{24}}\prod_{n=1}^\infty\bigg(\frac{1+q^{m(n-\frac{1}{2})}}
{1+q^{n-\frac{1}{2}}}\bigg)^2,
\end{eqnarray*}
and
\begin{eqnarray*}
&&\prod_{k=1}^{m-1}g_{(\frac{1}{2},\frac{1}{2}+\frac{k}{m})}(\tau)\\
&=&
\prod_{k=1}^{m-1}\bigg(
-q^{\frac{1}{2}\mathbf{B}_2(\frac{1}{2})}e^{\pi
i(\frac{1}{2}+\frac{k}{m})(-\frac{1}{2})}(1+q^\frac{1}{2}\zeta_m^k)
\prod_{n=1}^\infty(1+q^{n+\frac{1}{2}}\zeta_m^k)(1+q^{n-\frac{1}{2}}\zeta_m^{-k})\bigg)\\
&=&(-1)^{m-1}e^{\pi i\frac{1-m}{2}}q^\frac{1-m}{24}
\prod_{k=1}^{m-1}\prod_{n=1}^\infty(1+q^{n-\frac{1}{2}}\zeta_m^k)(1+q^{n-\frac{1}{2}}\zeta_m^{-k})\\
&=&(-1)^\frac{1-m}{2}q^\frac{1-m}{24}
\prod_{n=1}^\infty\prod_{k=1}^{m-1}(1+q^{n-\frac{1}{2}}\zeta_m^k)^2~\textrm{because
$m$ is odd}\\
&=&(-1)^\frac{1-m}{2}q^\frac{1-m}{24}
\prod_{n=1}^\infty\bigg(\frac{1+q^{m(n-\frac{1}{2})}}
{1+q^{n-\frac{1}{2}}}\bigg)^2~\textrm{by the identity
$\frac{1+X^m}{1+X}=\frac{1-(-X)^m}{1-(-X)}=\prod_{k=1}^{m-1}(1-(-X)\zeta_m^k)$}.
\end{eqnarray*}
This proves (iii).
\end{proof}

\section {Proof of Theorem \ref{main}}\label{proof1}

Let
\begin{equation}\label{Delta}
\Delta(\tau)=\eta(\tau)^{24}=(2\pi)^{12}q\prod_{n=1}^\infty(1-q^n)^{24}
\quad(\tau\in\mathfrak{H})
\end{equation}
be the \textit{modular discriminant function}.

\begin{proposition}\label{singular}
Let $\tau_0\in\mathfrak{H}$ be imaginary quadratic.
\begin{itemize}
\item[(i)] $j(\tau_0)$ is an algebraic integer.
\item[(ii)] Let $a$, $b$ and $d$ be integers with $ad>0$ and $\gcd(a,b,d)=1$.
Then, $a^{12}\Delta((a\tau_0+b)/d)/\Delta(\tau_0)$ is an algebraic
integer dividing $(ad)^{12}$.
\end{itemize}
\end{proposition}
\begin{proof}
See \cite{Lang} Chapter 5 Theorem 4 and Chapter 12 Theorem 4.
\end{proof}

\begin{proposition}\label{corollary}
Let $m$ be a positive integer and $\tau_0\in\mathfrak{H}$ be
imaginary quadratic.
\begin{itemize}
\item[(i)]
$\sqrt{m}\eta(m\tau_0)/\eta(\tau_0)$ is an algebraic integer
dividing $\sqrt{m}$.
\item[(ii)] $2g_{(\frac{1}{2},\frac{1}{2})}(m\tau_0)/
g_{(\frac{1}{2},\frac{1}{2})}(\tau_0)$ is an algebraic integer
dividing 4. In particular, if $m$ is odd, then
$g_{(\frac{1}{2},\frac{1}{2})}(m\tau_0)/g_{(\frac{1}{2},\frac{1}{2})}(\tau_0)$
is a unit.
\end{itemize}
\end{proposition}
\begin{proof}
(i) Applying Proposition \ref{singular}(ii) with $(a,b,d)=(m,0,1)$,
wee see that
\begin{equation*}
m^{12}\frac{\Delta(m\tau_0)} {\Delta(\tau_0)} =\bigg(\sqrt{m}
\frac{\eta(m\tau_0)}{\eta(\tau_0)}\bigg)^{24}
\end{equation*}
is an algebraic integer dividing $m^{12}$. We get the assertion by taking $24^\textrm{th}$ root.\\
(ii) We obtain from Proposition \ref{phiexpression}(ii) that
\begin{eqnarray*}
2\frac{g_{(\frac{1}{2},\frac{1}{2})}(m\tau_0)}
{g_{(\frac{1}{2},\frac{1}{2})}(\tau_0)}&=&e^{-\frac{\pi
i}{4}}g_{(0,\frac{1}{2})}(\tau_0) g_{(\frac{1}{2},0)}(\tau_0)
g_{(\frac{1}{2},\frac{1}{2})}(\tau_0)\frac{g_{(\frac{1}{2},\frac{1}{2})}(m\tau_0)}
{g_{(\frac{1}{2},\frac{1}{2})}(\tau_0)}\\
&=&e^{-\frac{\pi i}{4}}\bigg(g_{(0,\frac{1}{2})}(\tau_0)
g_{(\frac{1}{2},0)}(\tau_0)\bigg)
g_{(\frac{1}{2},\frac{1}{2})}(m\tau_0).
\end{eqnarray*}
By Propositions \ref{integrality}(i) and \ref{singular}(i), the
values $g_{(0,\frac{1}{2})}(\tau_0)g_{(\frac{1}{2},0)}(\tau_0)$,
$g_{(\frac{1}{2},\frac{1}{2})}(\tau_0)$,
$g_{(0,\frac{1}{2})}(m\tau_0)g_{(\frac{1}{2},0)}(m\tau_0)$ and
$g_{(\frac{1}{2},\frac{1}{2})}(m\tau_0)$ are algebraic integers.
Moreover, since
\begin{equation*}
\bigg(g_{(0,\frac{1}{2})}(\tau_0)g_{(\frac{1}{2},0)}(\tau_0)\bigg)
g_{(\frac{1}{2},\frac{1}{2})}(\tau_0)=
\bigg(g_{(0,\frac{1}{2})}(m\tau_0)g_{(\frac{1}{2},0)}(m\tau_0)\bigg)
g_{(\frac{1}{2},\frac{1}{2})}(m\tau_0)=2e^\frac{\pi i}{4}
\end{equation*}
by Proposition \ref{phiexpression}(ii), both
$g_{(0,\frac{1}{2})}(\tau_0) g_{(\frac{1}{2},0)}(\tau_0)$ and
$g_{(\frac{1}{2},\frac{1}{2})}(m\tau_0)$ are algebraic integers
dividing $2$. Hence
$2g_{(\frac{1}{2},\frac{1}{2})}(m\tau_0)/g_{(\frac{1}{2},\frac{1}{2})}(\tau_0)$
is an algebraic integer dividing $2\cdot2=4$.
\par
Now, suppose that $m$ ($\geq3$) is odd. Recall the relation
\begin{equation*}
\frac{g_{(\frac{1}{2},\frac{1}{2})}(m\tau)}{g_{(\frac{1}{2},\frac{1}{2})}(\tau)}
=(-1)^\frac{m-1}{2}\prod_{k=1}^{m-1}g_{(\frac{1}{2},\frac{1}{2}+\frac{k}{m})}(\tau)
\end{equation*}
given in Proposition \ref{phiexpression}(iii). Since each vector
$(\frac{1}{2},\frac{1}{2}+\frac{k}{m})$ has the composite primitive
denominator, $g_{(\frac{1}{2},\frac{1}{2}+\frac{k}{m})}(\tau)$ is a
modular unit over $\mathbb{Z}$ by Proposition \ref{integrality}(ii);
and hence so is
$g_{(\frac{1}{2},\frac{1}{2})}(m\tau)/g_{(\frac{1}{2},\frac{1}{2})}(\tau)$.
Therefore,
$g_{(\frac{1}{2},\frac{1}{2})}(m\tau_0)/g_{(\frac{1}{2},\frac{1}{2})}(\tau_0)$
is a unit by Proposition \ref{singular}(i).
\end{proof}

We are ready to prove Theorem \ref{main}. Let $m$ be a positive
integer and $\tau_0\in\mathfrak{H}$ be imaginary quadratic. By
Proposition \ref{phiexpression}(i) we can express
\begin{equation*}
2\sqrt{m}\frac{\varphi(m\tau_0)}{\varphi(\tau_0)}=
\sqrt{m}\frac{\eta(m\tau_0)}
{\eta(\tau_0)}\cdot2\frac{g_{(\frac{1}{2},\frac{1}{2})}(m\tau_0)}
{g_{(\frac{1}{2},\frac{1}{2})}(\tau_0)}.
\end{equation*}
Hence $2\sqrt{m}\varphi(m\tau_0)/\varphi(\tau_0)$ is an algebraic
integer dividing $4\sqrt{m}$ by Proposition \ref{corollary}(i) and
(ii). Similarly, if $m$ is odd, then
\begin{equation}\label{expression3}
\sqrt{m}\frac{\varphi(m\tau_0)}{\varphi(\tau_0)}=
\sqrt{m}\frac{\eta(m\tau_0)}
{\eta(\tau_0)}\cdot\frac{g_{(\frac{1}{2},\frac{1}{2})}(m\tau_0)}
{g_{(\frac{1}{2},\frac{1}{2})}(\tau_0)}
\end{equation}
is an algebraic integer dividing $\sqrt{m}$. This completes the
proof of Theorem \ref{main}.
\par
Now, we revisit and improve Theorem \ref{BCZ} as a corollary.

\begin{corollary}\label{cor1}
Let $m$ and $n$ be positive integers. If $m$ is odd, then
$\sqrt{m}\varphi(mni)/\varphi(ni)$ is an algebraic integer dividing
$\sqrt{m}$, while if $m$ is even, then
$2\sqrt{m}\varphi(mni)/\varphi(ni)$ is an algebraic integer dividing
$4\sqrt{m}$.
\end{corollary}
\begin{proof}
We get the assertion by setting $\tau_0=ni$ in Theorem \ref{main}.
\end{proof}

\begin{remark}
Berndt-Chan-Zhang used only Proposition \ref{singular}(ii) in order
to achieve Theorem \ref{BCZ}.
\end{remark}

\section{Proof of Theorem \ref{unitcriterion}}\label{proof2}

\begin{proposition}\label{morerelation}
Let $m$ ($\geq2$) be an integer.
\begin{itemize}
\item[(i)] We have the relation
\begin{equation*}
\prod_{\begin{smallmatrix}a,b\in\mathbb{Z}\\
0\leq a,b<m,~(a,b)\neq(0,0)\end{smallmatrix}}
g_{(\frac{a}{m},\frac{b}{m})}(\tau)^{12m}=m^{12m}.
\end{equation*}
\item[(ii)] We derive
\begin{equation*}
\prod_{k=1}^{m-1} g_{(0,\frac{k}{m})}(\tau)
=i^{m-1}\bigg(\sqrt{m}\frac{\eta(m\tau)}{\eta(\tau)}\bigg)^{2}.
\end{equation*}
\end{itemize}
\end{proposition}
\begin{proof}
(i) See \cite{K-L} p. 45 Example.\\
(ii) We deduce that
\begin{eqnarray*}
\prod_{k=1}^{m-1}g_{(0,~\frac{k}{m})}(\tau)&=&
\prod_{k=1}^{m-1}\bigg(-q^{\frac{1}{2}\mathbf{B}_2(0)}
\zeta_{2m}^{-k}(1-\zeta_m^k)\prod_{n=1}^\infty
(1-q^n\zeta_m^k)(1-q^n\zeta_m^{-k})\bigg)~\textrm{by the definition
(\ref{Siegel})}\\
&=&i^{m-1}mq^\frac{m-1}{12}\prod_{n=1}^\infty\bigg(\frac{1-q^{mn}}{1-q^n}\bigg)^2\\
&&\textrm{by the identity
$\frac{1-X^m}{1-X}=1+X+\cdots+X^{m-1}=\prod_{k=1}^{m-1}(1-X\zeta_m^k)$}\\
&=&i^{m-1}\bigg(\sqrt{m}\frac{\eta(m\tau)}{\eta(\tau)}\bigg)^2~\textrm{by
the definition (\ref{eta})}.
\end{eqnarray*}
\end{proof}
\begin{remark}
Let $\tau_0\in\mathfrak{H}$ be imaginary quadratic. By Propositions
\ref{integrality}(i), \ref{singular}(i) and \ref{morerelation}(i),
$\prod_{k=1}^{m-1} g_{(0,\frac{k}{m})}(\tau_0)$ is an algebraic
integer dividing $m$. It follows from Proposition
\ref{morerelation}(ii) that $\sqrt{m}\eta(m\tau_0)/\eta(\tau_0)$ is
an algebraic integer dividing $\sqrt{m}$. This gives another proof
of Proposition \ref{corollary}(i).
\end{remark}
From now on, we let $K$ be an imaginary quadratic field and
$\theta_K$ be as in (\ref{theta}). We denote $H_K$ and $K_{(N)}$ the
Hilbert class field and the ray class field modulo $N$ ($\geq1$) of
$K$, respectively.

\begin{proposition}[Main theorem of complex
multiplication]\label{CM} We have
\begin{equation*}
K_{(N)}=K\mathcal{F}_N(\theta_K)=K\bigg(h(\theta_K)~:~h\in\mathcal{F}_N~\textrm{is
defined and finite at $\theta_K$}\bigg).
\end{equation*}
\end{proposition}
\begin{proof}
See \cite{Lang} Chapter 10 Corollary to Theorem 2 or \cite{Shimura}
Chapter 6.
\end{proof}

\begin{corollary}
If $m$ ($\geq3$) is an odd integer, then
$(\sqrt{m}\varphi(m\theta_K)/\varphi(\theta_K))^{2}$ lies in
$K_{(48m^2)}$.
\end{corollary}
\begin{proof}
We see that
\begin{eqnarray}
\bigg(\sqrt{m}\frac{\varphi(m\tau)}{\varphi(\tau)}\bigg)^2
&=&\bigg(\sqrt{m}
\frac{\eta(m\tau)}{\eta(\tau)}\bigg)^2\bigg(\frac{g_{(\frac{1}{2},\frac{1}{2})}(m\tau)}
{g_{(\frac{1}{2},\frac{1}{2})}(\tau)}\bigg)^2~\textrm{by
Proposition \ref{phiexpression}(i)}\nonumber\\
&=&(-1)^\frac{1-m}{2}\prod_{k=1}^{m-1} g_{(0,\frac{k}{m})}(\tau)
g_{(\frac{1}{2},\frac{1}{2}+\frac{k}{m})}(\tau)^2
 ~\textrm{by
Propositions \ref{morerelation}(ii) and
\ref{phiexpression}(iii)}.\label{expression}
\end{eqnarray}
Hence $(\sqrt{m}\varphi(m\tau)/\varphi(\tau))^2$ belongs to
$\mathcal{F}_{48m^2}$ by Proposition \ref{integrality}(iii).
Therefore, $(\sqrt{m}\varphi(m\theta_K)/\varphi(\theta_K))^{2}$ lies
in $K_{(48m^2)}$ by Proposition \ref{CM}.
\end{proof}

\begin{proposition}\label{Rama}
If $N$ ($\geq2$) is an integer with more than one prime ideal factor
in $K$, then $g_{(0,\frac{1}{N})}(\theta_K)^{12N}$ is a unit which
lies in $K_{(N)}$.
\end{proposition}
\begin{proof}
See \cite{Ramachandra} $\S$6.
\end{proof}
\begin{remark}
In \cite{J-K-S} authors proved that if $K$ is an imaginary quadratic
field other than $\mathbb{Q}(\sqrt{-1})$ and $\mathbb{Q}\sqrt{-3})$,
then $g_{(0,\frac{1}{N})}(\theta_K)^{12N}$ is a primitive generator
of $K_{(N)}$ over $K$, which is called a \textit{Siegel-Ramachandra
invariant} (\cite{K-L} Chapter 11 $\S1$ or \cite{Ramachandra}).
\end{remark}

On the other hand, we have the following explicit description of the
Shimura's reciprocity law which connects the class field theory with
the theory of modular functions, due to Stevenhagen.

\begin{proposition}[Shimura's reciprocity law]\label{Stev}
Let $\mathrm{min}(\theta_K,~\mathbb{Q})=X^2+BX+C\in\mathbb{Z}[X]$.
For every positive integer $N$ the matrix group
\begin{equation*}W_{K,N}=\bigg\{\begin{pmatrix}t-Bs &
-Cs\\s&t\end{pmatrix}\in\mathrm{GL}_2(\mathbb{Z}/N\mathbb{Z})~:~t,~s\in\mathbb{Z}/N\mathbb{Z}\bigg\}
\end{equation*}
gives rise to the surjection
\begin{eqnarray}\label{surj}
W_{K,N}&\longrightarrow&\mathrm{Gal}(K_{(N)}/H_K)\\
\alpha&\mapsto&\bigg(h(\theta)\mapsto
h^\alpha(\theta_K)\bigg)\nonumber
\end{eqnarray}
where $h\in\mathcal{F}_N$ is defined and finite at $\theta_K$. Its
kernel is given by
\begin{equation}\label{kernel}
\left\{\begin{array}{ll}
\bigg\{\pm\begin{pmatrix}1&0\\0&1\end{pmatrix},~
\pm\begin{pmatrix}0&-1\\1&0\end{pmatrix}
\bigg\} & \textrm{if $K=\mathbb{Q}(\sqrt{-1})$}\vspace{0.2cm}\\
\bigg\{\pm\begin{pmatrix}1&0\\0&1\end{pmatrix},~
\pm\begin{pmatrix}-1&-1\\1&0\end{pmatrix},~
\pm\begin{pmatrix}0&-1\\1&1\end{pmatrix}
\bigg\} & \textrm{if $K=\mathbb{Q}(\sqrt{-3})$}\vspace{0.2cm}\\
\bigg\{\pm\begin{pmatrix}1&0\\0&1\end{pmatrix}\bigg\} &
\textrm{otherwise.}
\end{array}\right.
\end{equation}
\end{proposition}
\begin{proof}
See \cite{Stevenhagen} $\S$3.
\end{proof}

\begin{proposition}\label{etaextend}
If $m$ ($\geq2$) is an integer whose prime factors split in $K$,
then $\sqrt{m}\eta(m\theta_K)/\eta(\theta_K)$ is a unit.
\end{proposition}
\begin{proof}
We get from Proposition \ref{morerelation}(ii) that
\begin{equation}\label{expression2}
\bigg(\sqrt{m}\frac{\eta(m\theta_K)}{\eta(\theta_K)}\bigg)^{24m}
=\prod_{k=1}^{m-1}g_{(0,\frac{k}{m})}(\theta_K)^{12m}.
\end{equation}
For each $1\leq k\leq m-1$, let us write
\begin{equation*}\frac{k}{m}
=\frac{a}{b}~\textrm{with positive integers $a$ and $b$ such that
$\gcd(a,b)=1$}.
\end{equation*}
Since $g_{(0,\frac{1}{b})}(\theta_K)^{12b}$ lies in $K_{(b)}$ by
Propositions \ref{integrality}(iii) and \ref{CM}, and
$\begin{pmatrix}a&0\\0&a\end{pmatrix}\in
W_{K,b}\simeq\mathrm{Gal}(K_{(b)}/H_K)$, we derive that
\begin{eqnarray*}
\bigg(g_{(0,\frac{1}{b})}(\theta_K)^{12b}\bigg)^{\left(\begin{smallmatrix}
a&0\\0&a\end{smallmatrix}\right)}&=&
\bigg(g_{(0,\frac{1}{b})}(\tau)^{12b}\bigg)^{\left(\begin{smallmatrix}
a&0\\0&a\end{smallmatrix}\right)}(\theta_K)~ \textrm{by Proposition
\ref{Stev}}\\
&=&
\bigg(g_{(0,\frac{1}{b})\left(\begin{smallmatrix}a&0\\0&a\end{smallmatrix}\right)}(\tau)^{12b}\bigg)(\theta_K)
~\textrm{by Proposition \ref{transform}(iv)}\\
&=&g_{(0,\frac{a}{b})}(\theta_K)^{12b}.
\end{eqnarray*}
On the other hand, since $b$ has more than one prime ideal factor in
$K$ by the assumption on $m$, $g_{(0,\frac{1}{b})}(\theta_K)^{12b}$
is a unit by Proposition \ref{Rama}. Hence
$g_{(0,\frac{k}{m})}(\theta_K)^{12m}=
(g_{(0,\frac{a}{b})}(\theta_K)^{12b})^{m/b}$ is also a unit.
Therefore $\sqrt{m}{\eta(m\theta_K)}{\eta(\theta_K)}$ becomes a unit
by the relation (\ref{expression2}).
\end{proof}

Now, we can prove Theorem \ref{unitcriterion}. Let $m$ ($\geq3$) be
an odd integer whose prime factors split in $K$. Since both
$\sqrt{m}\eta(m\theta_K)/\eta(\theta_K)$ and
$g_{(\frac{1}{2},\frac{1}{2})}(m\theta_K)/g_{(\frac{1}{2},\frac{1}{2})}(\theta_K)$
are units by Propositions \ref{etaextend} and \ref{corollary}(ii),
the result follows from the expression (\ref{expression3}) with
$\tau_0=\theta_K$. This completes the proof.

\begin{corollary}\label{cor2}
Let $m$ ($\geq3$) be an odd integer whose prime factors $p$ satisfy
$p\equiv1\pmod{4}$. Then $\sqrt{m}\varphi(mi)/\varphi(i)$ is a unit.
\end{corollary}
\begin{proof}
If $K=\mathbb{Q}(\sqrt{-1})$, then $\theta_K=i$. For each prime
factor $p$ of $m$, the fact $p\equiv1\pmod{4}$ implies that $p$
splits in $K$ (\cite{Cox} Corollary 5.17). We get the assertion by
applying Theorem \ref{unitcriterion}.
\end{proof}

We close this section by evaluating $\sqrt{m}\varphi(mi)/\varphi(i)$
for $m=3$ and $5$, explicitly.

\begin{example}
We shall evaluate $\sqrt{3}\varphi(3i)/\varphi(i)$. If
$K=\mathbb{Q}(\sqrt{-1})$, then $\theta_K=i$ and $H_K=K$ (\cite{Cox}
Theorem 12.34). By Proposition \ref{Stev} we have
\begin{eqnarray*}
\mathrm{Gal}(K_{(6)}/K)&\simeq&W_{K,6}/
\bigg\{\pm\begin{pmatrix}1&0\\0&1\end{pmatrix},~
\pm\begin{pmatrix}0&-1\\1&0\end{pmatrix} \bigg\}\\
&=&\bigg\{\alpha_1=\begin{pmatrix}1&0\\0&1\end{pmatrix},~
\alpha_2=\begin{pmatrix}1&-2\\2&1\end{pmatrix},~
\alpha_3=\begin{pmatrix}1&-4\\4&1\end{pmatrix},~
\alpha_4=\begin{pmatrix}3&-2\\2&3\end{pmatrix}\bigg\}.
\end{eqnarray*}
Since
\begin{eqnarray*}
x=\bigg(\sqrt{3}\frac{\varphi(3i)}{\varphi(i)}\bigg)^{24}&=&g_{(0,\frac{1}{3})}(i)^{12}g_{(0,\frac{2}{3})}(i)^{12}
g_{(\frac{1}{2},\frac{5}{6})}(i)^{24}g_{(\frac{1}{2},\frac{7}{6})}(i)^{24}
~\textrm{by (\ref{expression})}\\
&=&g_{(0,\frac{1}{3})}(i)^{24}g_{(\frac{1}{2},\frac{1}{6})}(i)^{48}
~\textrm{by Proposition \ref{transform}(i) and (iii)}\\
&\approx& 72954,
\end{eqnarray*}
$x$ lies in $K_{(6)}$ by Propositions \ref{integrality}(iii) and
\ref{CM}. Hence its conjugates $x_k=x^{\alpha_k}$ ($1\leq k\leq4$)
over $K$ are
\begin{eqnarray*}
x_1&=&g_{(0,\frac{1}{3})}(i)^{24}g_{(\frac{1}{2},\frac{1}{6})}(i)^{48},\\
x_2&=&g_{(\frac{2}{3},\frac{1}{3})}(i)^{24}g_{(\frac{5}{6},\frac{1}{6})}(i)^{48},\\
x_3&=&g_{(\frac{1}{3},\frac{1}{3})}(i)^{24}g_{(\frac{1}{6},\frac{1}{6})}(i)^{48},\\
x_4&=&g_{(\frac{2}{3},0)}(i)^{24}g_{(\frac{5}{6},\frac{1}{2})}(i)^{48}
\end{eqnarray*}
with some multiplicity by Propositions \ref{Stev} and
\ref{transform}(iv).
 We claim that the minimal polynomial of $x$
over $K$ has integeral coefficients. Indeed, since $x$ is a real
algebraic integer by the definition (\ref{phi}) and Theorem
\ref{main}, we have
\begin{equation*}
[\mathbb{Q}(x):\mathbb{Q}]=\frac{[K(x):K]\cdot[K:\mathbb{Q}]}{[K(x):\mathbb{Q}(x)]}
=\frac{[K(x):K]\cdot2}{2}=[K(x):K],
\end{equation*}
from which the claim follows. Thus $x$ is a zero of the polynomial
\begin{equation*}
(X-x_1)(X-x_2)(X-x_3)(X-x_4)=(X^2-72954X+729)^2
\end{equation*}
whose coefficients are determined by numerical approximation.
Therefore we obtain
\begin{equation*}
\sqrt{3}\frac{\varphi(3i)}{\varphi(i)}=\sqrt[24]{x}
=\sqrt[24]{36477+21060\sqrt{3}}=\sqrt[4]{3+2\sqrt{3}}.
\end{equation*}
\end{example}

\begin{example}
Now, we consider $\sqrt{5}\varphi(\sqrt{5}i)/\varphi(i)$. Let
$K=\mathbb{Q}\sqrt{-1})$. By Proposition \ref{Stev} we have
\begin{eqnarray*}
W_{K,10}&\simeq&\bigg\{\alpha_1=\begin{pmatrix}1&0\\0&1\end{pmatrix},~
\alpha_2=\begin{pmatrix}1&-4\\4&1\end{pmatrix},~
\alpha_3=\begin{pmatrix}1&-6\\6&1\end{pmatrix},~
\alpha_4=\begin{pmatrix}2&-3\\3&2\end{pmatrix},\\
&&\alpha_5=\begin{pmatrix}2&-5\\5&2\end{pmatrix},~
\alpha_6=\begin{pmatrix}2&-7\\7&2\end{pmatrix},~
\alpha_7=\begin{pmatrix}3&0\\0&3\end{pmatrix},~
\alpha_8=\begin{pmatrix}4&-5\\5&4\end{pmatrix}\bigg\}.
\end{eqnarray*}
Since
\begin{eqnarray*}
x=\bigg(\sqrt{5}\frac{\varphi(5i)}{\varphi(i)}\bigg)^{120}&=&
g_{(0,\frac{1}{5})}(i)^{120} g_{(0,\frac{2}{5})}(i)^{120}
g_{(\frac{1}{2},\frac{1}{10})}(i)^{240}
g_{(\frac{1}{2},\frac{3}{10})}(i)^{240}
~\textrm{by Proposition \ref{transform}(i) and (iii)}\\
&\approx&41473935220454921602871195774259272002,
\end{eqnarray*}
$x$ lies in $K_{(10)}$ by Propositions \ref{integrality}(iii) and
\ref{CM}.  Hence its conjugates $x_k=x^{\alpha_k}$ ($1\leq k\leq8$)
over $K$ are
\begin{eqnarray*}
x_1=x_5=x_7=x_8&=&g_{(0,\frac{1}{5})}(i)^{120}
g_{(0,\frac{2}{5})}(i)^{120} g_{(\frac{1}{2},\frac{1}{10})}(i)^{240}
g_{(\frac{1}{2},\frac{3}{10})}(i)^{240},\\
x_2&=&g_{(\frac{4}{5},\frac{1}{5})}(i)^{120}
g_{(\frac{3}{5},\frac{2}{5})}(i)^{120}
g_{(\frac{9}{10},\frac{1}{10})}(i)^{240}
g_{(\frac{7}{10},\frac{3}{10})}(i)^{240},\\
x_3=x_6&=&g_{(\frac{1}{5},\frac{1}{5})}(i)^{120}
g_{(\frac{2}{5},\frac{2}{5})}(i)^{120}
g_{(\frac{1}{10},\frac{1}{10})}(i)^{240}
g_{(\frac{3}{10},\frac{3}{10})}(i)^{240},\\
x_4&=&g_{(\frac{3}{5},\frac{2}{5})}(i)^{120}
g_{(\frac{1}{5},\frac{4}{5})}(i)^{120}
g_{(\frac{3}{10},\frac{7}{10})}(i)^{240}
g_{(\frac{9}{10},\frac{1}{10})}(i)^{240}
\end{eqnarray*}
with some multiplicity by Propositions \ref{Stev} and
\ref{transform}(iv). So $x$ is a zero of the polynomial
\begin{equation*}
(X^2-41473935220454921602871195774259272002X+1)^4,
\end{equation*}
which illustrates that $x$ is a unit. Therefore we get
\begin{eqnarray*}
\sqrt{5}\frac{\varphi(5i)}{\varphi(i)}&=&\sqrt[120]{x}\\
&=&\sqrt[120]{20736967610227460801435597887129636001+9273853844735993106095069260699853880\sqrt{5}}\\
&=&\sqrt[10]{682+305\sqrt{5}}.
\end{eqnarray*}
\end{example}

\bibliographystyle{amsplain}

\end{document}